\documentclass{amsart}
\usepackage{amssymb}

\newcommand{\IN}{\mathbb N}
\newcommand{\IR}{\mathbb R}
\newcommand{\IZ}{\mathbb Z}
\newcommand{\IT}{\mathbb T}
\newcommand{\conv}{\mathrm{conv}}
\newcommand{\pr}{\mathrm{pr}}
\newcommand{\e}{\varepsilon}
\newcommand{\w}{\omega}

\newtheorem{theorem}{Theorem}
\newtheorem{claim}{Claim}

\newtheorem{example}{Example}

\title{A simple inductive proof of Lev\'y-Steinitz theorem}
\author{Taras Banakh}
\address{Jan Kochanowski University in Kielce (Poland) and Ivan Franko National University of Lviv (Ukraine)}
\email{t.o.banakh@gmail.com}
\begin{document}

\begin{abstract} 
We present a relatively simple inductive proof of the classical Lev\'y-Steinitz Theorem saying that for a sequence $(x_n)_{n=1}^\infty$ in a finite-dimensional Banach space $X$ the set  $\Sigma$ of all sums of rearranged series $\sum_{n=1}^\infty x_{\sigma(n)}$ is an affine subspace of $X$ such that $\Sigma=\Sigma+\Gamma^\perp$ where $\Gamma=\{f\in X^*:\sum_{n=1}^\infty|f(x_n)|<\infty\}$ and $\Gamma^\perp=\bigcap_{f\in\Gamma}f^{-1}(0)$. The affine subspace $\Sigma$ is not empty if and only if for any linear functional $f:X\to \IR$ the series $\sum_{n=1}^\infty f(x_{\sigma(n)})$ is convergent for some permutation $\sigma$ of $\mathbb N$. This answers a problem of Vaja Tarieladze, posed in Lviv Scottish Book in September, 2017.

Also we construct a sequence $(x_n)_{n=1}^\infty$ in the torus $\IT\times\IT$ such that the series $\sum_{n=1}^\infty x_{\sigma(n)}$ is divergent for all permutations $\sigma$ of $\IN$ but for any continuous homomorphism $f:\IT^2\to\IT$ to the circle group $\IT:=\IR/\IZ$ the series $\sum_{n=1}^\infty f(x_{\sigma_f(n)})$ is convergent for some permutation $\sigma_f$ of $\mathbb N$. This example 
shows that the second part of Lev\'y-Steinitz Theorem (characterizing sequences with non-empty set of potential sums) does not extend to locally compact Abelian groups.\end{abstract}
\maketitle

\section{Introduction}

In this paper we present a simple inductive proof of the famous Lev\'y-Steinitz theorem on sums of rearranged series in finite-dimensional Banach spaces.

We shall say that a series $\sum_{n=1}^\infty x_n$ in a Banach space $X$ is {\em potentially convergent} to a point $x\in X$ if for some permutation $\sigma$ of $\mathbb N$ the sequence of partial sums $\big(\sum_{n=1}^m x_{\sigma(n)}\big)_{m=1}^\infty$ converges to $x$.
In this case we shall say that $x$ is a {\em potential sum} of the series $\sum_{n=1}^\infty x_n$.

A series $\sum_{n=1}^\infty x_n$ in a Banach space $X$ is called {\em $*$-potentially convergent} if for any linear continuous functional $f\in X^*$ the series $\sum_{n=1}^\infty f(x_n)$ is potentially convergent to some real number. Here by $X^*$ we denote the dual Banach space to $X$.

A subset $A$ of a linear space $X$ is called {\em affine} if $\sum_{i=1}^nt_ia_i\in A$ for any $n\in\IN$, points $a_1,\dots,a_n\in A$ and real numbers $t_1,\dots,t_n$ with $\sum_{i=1}^nt_i=1$. According to this definition, the empty subset of $X$ is affine.

\begin{theorem}[Lev\'y, Steinitz] For a sequence $(x_n)_{n\in\w}$ of points of a finite-dimensional Banach space $X$ the set $\Sigma$ of potential sums of the series $\sum_{n=1}^\infty x_n$ is an affine subspace of $X$ such that $\Sigma=\Sigma+\Gamma^\perp$ where $\Gamma=\{f\in X^*:\sum_{n=1}^\infty|f(x_n)|<\infty\}$ and $\Gamma^\perp=\bigcap_{f\in\Gamma}f^{-1}(0)$. The set $\Sigma$ is not empty if and only if the series $\sum_{n=1}^\infty x_n$ is $*$-potentially convergent.
\end{theorem}

Lev\'y-Steinitz Theorem was initially proved by Lev\'y \cite{Levy} in 1905. But the proof was very complicated and contained a gap in the higher-dimensional case, which was filled by Steinitz \cite{Steinitz} in 1913. The proof of Steinitz also was very long and it was eventually clarified and simplified by Gross \cite{Gross}, Halperin \cite{Halp}, Rosenthal \cite{Rosen}, Kadets and Kadets \cite[\S2.1]{KK}. But even after simplifications, the proof remained too complicated (8 pages in \cite{Rosen}). This situation motivated Vaja Tarieladze to pose a problem in Lviv Scottish Book \cite{Tar} of finding a simple and transparent proof of  Lev\'y-Steinitz Theorem or at least of its second part (characterizing series with non-empty set of potential sums).
Inspired by this question of Tarieladze, in this paper we present a simple inductive proof of Lev\'y-Steinitz Theorem.

\section{Proof of Lev\'y-Steinitz Theorem}

The proof of the theorem is by induction on the dimension of the Banach space $X$.
For zero-dimensional Banach spaces the theorem is trivial. Assume that  the theorem is true for all Banach spaces of dimension less than some number $d\in\IN$. Let $X$ be a Banach space of dimension $d$, $\sum_{n=1}^\infty x_n$ be a series in $X$, and $\Sigma$ be the set of all potential sums of this series in $X$. If the series $\sum_{n=1}^\infty x_n$ is not $*$-potentially convergent, then it not potentially convergent and the set $\Sigma$ is empty.

So, we assume that the series $\sum_{n=1}^\infty x_n$ is $*$-potentially convergent. In this case we shall prove that $\Sigma$ is a non-empty affine subset of $X$ and $\Sigma=\Sigma+\Gamma^\perp$ where $\Gamma=\{f\in X^*:\sum_{n=1}^\infty|f(x_n)|<\infty\}$.

Without loss of generality, each $x_n$ is not equal to zero.
Also we can assume that the norm of $X$ is generated by a scalar product $\langle\cdot,\cdot\rangle$, so $X$ is a finite-dimensional Hilbert space, whose dual $X^*$ can be identified with $X$. Under such identification, the set $\Gamma$ coincides with the set $\{y\in X:\sum_{n=1}^\infty|\langle y,x_n\rangle|<\infty\}$. Let $S=\{x\in X:\|x\|=1\}$ denote the unit sphere of the Banach space $X$.

\begin{claim}\label{cl1} The sequence $(x_n)_{n=1}^\infty$ tends to zero.
\end{claim}

\begin{proof} Assuming that $(x_n)_{n=1}^\infty$ does not tend to zero, we can find an $\varepsilon>0$ such that the set $E=\{n\in\IN:\|x_n\|\ge\e\}$ is infinite. By the compactness of the unit sphere $S=\{x\in X:\|x\|=1\}$ the sequence $\{x_n/\|x_n\|\}_{n\in E}\subset S$ has an accumulating point $x_\infty\in S$. Then for the linear functional $f:X\to\IR$, $f:x\mapsto\langle x,x_\infty\rangle$, the series $\sum_{n=1}^\infty f(x_n)$ is not potentially convergent as its terms do no tend to zero.
But this contradicts our assumption (that $\sum_{n=1}^\infty x_n$ is $*$-potentially convergent).
\end{proof}

%For a subset $U$ of the unit sphere $S$ of $X$ and any $n\in\IN$ let
%$$x_n^U:=\begin{cases}x_n&\mbox{if $\frac{x_n}{\|x_n\|}\in U$};\\
%0&\mbox{otherwise}.
%\end{cases}
%$$
 A point $x$ of the sphere $S$ is defined to be a {\em divergence direction} of the series $\sum_{n=1}^\infty x_n$ if for every neighborhood $U\subset S$ of $x$ the set $$\IN_U:=\{n\in\mathbb N:\frac{x_n}{\|x_n\|}\in U\}$$ is infinite and the series $\sum_{n\in\IN_U}\|x_n\|$ is divergent.

Let $D\subset S$ be the (closed) set of all divergence directions of the series $\sum_{n=1}^\infty x_n$. If $D$ is empty, then the (compact) sphere $S$ admits a finite cover $\mathcal U$ by open subsets $U\subset S$ for which the series $\sum_{n\in\IN_U}\|x_n\|$ is convergent. Then $\sum_{n=1}^\infty\|x_n\|\le\sum_{U\in\mathcal U}\sum_{n\in\IN_U}\|x_n\|<\infty$, so the series $\sum_{n=1}^\infty x_n$ is absolutely convergent, $\Gamma=X$ and the set $\Sigma$ is a singleton, equal to $\Sigma=\Sigma+\{0\}=\Sigma+\Gamma^\perp$.

So, assume that the (compact) set $D$ is not empty. We claim that the convex hull of $D$ contains zero. In the opposite case we can apply the Hahn-Banach Theorem and find a linear functional $f:X\to \IR$ such that $f(D)\subset(0,+\infty)$ hence the series $\sum_{n=1}^\infty f(x_n)$ fails to converge potentially, which contradicts our assumption.

This contradiction shows that zero belongs to the convex hull of the set $D$. Let $D_0$ be a subset of smallest (finite) cardinality containing zero in its convex hull $\conv(D_0)$ and let $X_0$ be the linear hull of $D_0$ in $X$. The minimality of $D_0$ ensures that the set $D_0$ is affinely independent and zero is contained in the interior of the convex hull $\conv(D_0)$ in $X_0$. Let $X_1:=\bigcap_{y\in X_0}\{x\in X:\langle x,y\rangle=0\}$ be the orthogonal complement of $X_0$ in $X$, and let $\pr_0:X\to X_0$ and $\pr_1:X\to X_1$ be the orthogonal projections.

\begin{claim}\label{cl2} For every $z\in D_0$ there exists a subset $\Omega_z\subset\IN$ such that
\begin{enumerate}
\item $\sum_{n\in\Omega_z}\|\pr_1(x_n)\|<\infty$;
\item for each neighborhood $U\subset S$ of $z$ the series $\sum_{n\in\Omega_z\cap\IN_U}\|x_n\|$ is divergent.
\end{enumerate}
\end{claim}

\begin{proof} For every $k\in\w$ consider the neighborhood $O_k=\{x\in S:\|x-z\|<\frac1{2^k}\}$ of $z$ in $S$ and observe that each point $x\in O_k$ has $$\|\pr_1(x)\|=\|\pr_1(x-z)+\pr_1(z)\|=
\|\pr_1(x-z)+0\|\le\|x-z\|<\frac1{2^k}.$$
Since $z\in D_0\subset D$, the set $\IN_{O_k}=\{n\in\IN:\frac{x_n}{\|x_n\|}\in O_k\}$ is infinite and the series $\sum_{n\in\IN_{O_k}}\|x_n\|$ is divergent. Using Claim~\ref{cl1}, we can find a finite subset $F_k\subset\IN_{O_k}$ such that $$k<\sum_{n\in F_k}\|x_n\|<k+1.$$

We claim that the set $\Omega_z=\bigcup_{k\in\w}F_k$ has the required properties. Indeed,
$$
\begin{aligned}
&\sum_{n\in\Omega_z}\|\pr_1(x_n)\|\le\sum_{k=0}^\infty\sum_{n\in F_k}\|\pr_1(x_n)\|=\sum_{k=0}^\infty\sum_{n\in F_k}
\big\|\pr_1(\tfrac{x_n}{\|x_n\|})\big\|\cdot \|x_n\|<\\
&<\sum_{k=0}^\infty\sum_{n\in F_k}\frac1{2^k}\|x_n\|=\sum_{k=0}^\infty\frac1{2^k}\sum_{n\in F_k}\|x_k\|<\sum_{k=0}^\infty\frac{k+1}{2^k}<\infty.
\end{aligned}
$$
So, the first condition is satisfied. To check the second condition, take any neighborhood $U\subset S$ of $z$ and find $k\in\w$ such that $O_k\subset U$. The latter inclusion implies $F_m\subset \IN_{O_n}\subset\IN_U$ for all $m\ge k$ and hence $$
\sum_{n\in\IN_U}\|x_n\|\ge \sup_{m\ge k}\sum_{n\in F_m}\|x_n\|\ge\sup_{m\ge k}m=\infty,$$
so the series $\sum_{n\in\IN_U}\|x_n\|$ is divergent.
\end{proof}

Claim~\ref{cl2} implies that the set $\Omega=\bigcup_{z\in D_0}\Omega_z$ has the following properties:
\begin{enumerate}
\item[$(\Omega1)$] $\sum_{n\in\Omega}\|\pr_1(x_n)\|<\infty$;
\item[$(\Omega2)$] for each $x\in D_0$ and each neighborhood $U\subset S$ of $x$ the series $\sum_{n\in\Omega\cap\IN_U}\|x_n\|$ is divergent.
\end{enumerate}
%\end{document}

\begin{claim}\label{cl3} There exists a positive constant $C$ such that for every $x\in X_0$, $\e>0$ and finite set $F\subset\Omega$, there exists a finite set $E\subset \Omega\setminus F$ such that
\begin{enumerate}
\item $\|x-\sum_{n\in E}x_n\|<\e$ and
\item $\big\|\sum_{n\in E'}x_n\big\|\le C\cdot\max\{\|x\|,\e\}$ for any subset $E'\subset E$.
\end{enumerate}
\end{claim}

\begin{proof} Since the interior of the convex hull of the set $D_0$  contains zero, it also contains some closed ball $\{x\in X_0:\|x\|\le\delta\}$ of positive radius $\delta<\frac14$. Replacing $\delta$ by a smaller number, we can assume that $\|x-y\|>2\delta$ for any distinct points $x,y\in D_0$. We claim that the constant $C=\frac{C_\delta}{\delta}$ where $$C_\delta=\sup\big\{\big\|\sum_{z\in D_0}t_zx_z\big\|:\forall z\in D_0\;(t_z\in[0,1],\;\|x_z-z\|<\delta)\big\}$$
 satisfies our requirements.

Indeed, fix any  $x\in X$, $\e>0$, and a finite set $F\subset\Omega$.
Replacing $\e$ by a smaller number, we can assume that $\e<\delta<\frac14$.
 Let $$c:=\frac1\delta\cdot\max\{\|x\|,\e\}.$$

It follows that $\|\frac{x}c\|\le\delta$ and hence the point $\frac{x}c$ belongs to the interior of the convex hull of the affinely independent set $D_0$ in $X_0$. So,  $\frac{x}c=\sum_{z\in D_0}t_zz$ and hence  $x=\sum_{z\in D_0}ct_zz$ for some sequence $(t_z)_{z\in D_0}$ of positive real numbers with $\sum_{z\in D_0}t_z=1$.

For every $z\in D_0$ consider the spherical disk $S_z=\big\{x\in S:\|x-z\|<\min\{\delta,\frac\e{2c}\}\big\}$ and its convex cone $\hat S_z=\{t\cdot s:t>0,\;s\in S_z\}$. Since $\delta<\frac12\|z-z'\|$ for any distinct points $z,z'\in D_0$, the cones $\hat S_z$, $z\in D_0$, are pairwise disjoint. Observe also that for any elements $x,y\in S_z$, we have the lower bound $$\langle x,y\rangle=\langle x,x\rangle+\langle x,y-x\rangle\ge 1-\|x-y\|>1-2\delta\ge \frac12.$$ Then for any elements $x,y\in\hat S_z$ we have the inequality
\begin{equation}\label{sum}
\|x+y\|\ge \langle x+y,\tfrac{x}{\|x\|}\rangle=\|x\|+{\|y\|}\langle
\tfrac{y}{\|y\|},\tfrac{x}{\|x\|}\rangle \ge\|x\|+\frac12\|y\|.
\end{equation}

 By the property $(\Omega2)$ of the set $\Omega$, for every $z\in D_0$ the set $\Omega_z:=\{n\in\Omega:\frac{x_n}{\|x_n\|}\in S_z\}=\{n\in\Omega:x_n\in\hat S_z\}$ is infinite and the series $\sum_{n\in\Omega_z}\|x_n\|$ is divergent. Since the sequence $(x_n)_{n\in\Omega_z}$ is contained in the cone $\hat S_z$, tends to zero, and the series $\sum_{n\in\Omega_z}\|x_n\|$ is divergent, we can use the inequality (\ref{sum}) and choose a finite subset $E_z\subset \Omega_z\setminus F$ such that the sum $s_z=\sum_{n\in E_z}x_n\in \hat S_z$ has norm in the interval $$c\cdot t_z-\frac{\e}{2|D_0|}<\|s_z\|\le c\cdot t_z$$
where $|D_0|$ denotes the cardinality of the finite set $D_0$.

Then
$$
\|s_z-ct_zz\|\le
\Big\|\|s_z\|\cdot\frac{s_z}{\|s_z\|}-ct_z\frac{s_z}{\|s_z\|}\Big\|+
\Big\|ct_z\frac{s_z}{\|s_z\|}-ct_zz\Big\|=
\Big|\|s_z\|-ct_z\Big|+ct_z\Big\|\frac{s_z}{\|s_z\|}-z\Big\|<\frac{\e}{2|D_0|}+ct_z\frac{\e}{2c}
$$
and for the finite set $E=\bigcup_{z\in D_0}E_z\subset\Omega\setminus F$ we get
$$
\begin{aligned}
\Big\|\sum_{n\in E}x_n-x\Big\|&=\Big\|\sum_{n\in E}x_n-\sum_{z\in D_0}ct_zz\Big\|
\le \sum_{z\in D_0}\Big\|\sum_{n\in E_z}x_n-ct_zz\Big\|=\sum_{z\in D_0}\|s_z-ct_zz\|<\\
&<\sum_{z\in D_0}\big(\frac{\e}{2|D_0|}+t_z\frac{\e}2\big)=\frac\e2(1+\sum_{z\in D_0}{t_z})=\e.
\end{aligned}
$$
Now we prove that the set $E$ satisfies the second condition of Claim~\ref{cl3}. Take any subset $E'\subset E$ and for every $z\in D_0$ let $E'_z=E\cap E_z$. Then $\sum_{n\in E'_z}x_n\in \hat S_z$ and $\|\sum_{n\in E_z'}x_n\|\le\|\sum_{n\in E_z}x_n\|\le ct_z\le c$, which implies that $\sum_{n\in E'_z}x_n\in[0,c]\cdot S_z$ and thus
$$\|\sum_{n\in E'}x_n\|\le c C_\delta=C_\delta\cdot \frac{\max\{\e,\|x\|\}}{\delta}=C\cdot\max\{\|x\|,\e\}$$ by definition of the numbers $C_\delta$ and $C=C_\delta/\delta$.
\end{proof}

\begin{claim}\label{cl4} The sets $\Gamma=\{y\in X:\sum_{n=1}|\langle y,x_n\rangle|<\infty\}$ and $\Gamma_1=\{y\in X_1:\sum_{n=1}|\langle y,\pr_1(x_n)\rangle|<\infty\}$ coincide.
\end{claim}

\begin{proof} First we show that $\Gamma\subset X_1=X_0^\perp$. In the opposite case we would find points $y\in\Gamma$ and $z\in D_0$ such that $\langle y,z\rangle\ne 0$. Then $U=\{x\in S:|\langle y,x\rangle|>\frac12|\langle y,z\rangle|\}$ is an open neighborhood of $z$ in the sphere $S$ and by the definition of the set $D\ni z$ the set $\IN_U=\{n\in\IN:\frac{x_n}{\|x_n\|}\in U\}$ is infinite and the series $\sum_{n\in\IN_U}\|x_n\|$ diverges. It follows that for every $n\in\IN_U$ we have $|\langle y,x_n\rangle|>\frac{\|x_n\|}2|\langle y,z\rangle|$, which implies the divergence of the series $\sum_{n\in \IN_U}|\langle y,x_n\rangle|$ but this contradicts the choice of $y\in\Gamma$. This contradiction shows that $\Gamma\subset X_1$.

Next, observe that for any $y\in X_1$ and $n\in\IN$ we get
$$\langle y,x\rangle=\langle y,x-\pr_1(x)\rangle+\langle y,\pr_1(x)\rangle=\langle y,\pr_0(x)\rangle+\langle y,\pr_1(x)\rangle=0+\langle y,\pr_1(x)\rangle=\langle y,\pr_1(x)\rangle$$and thus $\sum_{n=1}^\infty|\langle y,x\rangle|=\sum_{n=1}^\infty|\langle y,\pr_1(x)\rangle|$, which implies the equality $\Gamma_1=X_1\cap \Gamma=\Gamma$.
\end{proof}

Now we are ready to complete the inductive proof of the theorem. By our assumption, the series $\sum_{n=1}^\infty x_n$ is $*$-potentially convergent. Then the series $\sum_{n=1}^\infty \pr_1(x_n)$ is $*$-potentially convergent in the Hilbert space $X_1$.

Since the space $X_1$ has dimension strictly smaller than $d$, the induction hypothesis and Claim~\ref{cl4} guarantee that the set $\Sigma_1$ of potential sums of the series $\sum_{n\in\IN}\pr_1(x_n)$ is a non-empty affine subspace of $X_1$ such that $\Sigma_1=\Sigma_1+\Gamma_1^\perp$, where $\Gamma_1^\perp=\{x\in X_1:\forall y\in\Gamma_1\;\langle y,x\rangle=0\}$ and $\Gamma_1=\Gamma=\{y\in X:\sum_{n=1}^\infty|\langle y,x_n\rangle|$ according to Claim~\ref{cl4}.

It remains to prove that the set $\Sigma$ of potential sums of the series $\sum_{n=1}^\infty x_n$ coincides with the non-empty affine subspace $X_0\oplus \Sigma_1$ of $X$.

 The inclusion $\Sigma\subset X_0\oplus \Sigma_1$ is trivial. To prove the reverse inclusion, we need to show that for any points $x\in X_0$ and $y\in\Sigma_1$ there exists a permutation $\sigma$ of $\mathbb N$ such that $\sum_{n=1}^\infty x_{\sigma(n)}=x+y$.

It will be more convenient instead of permutation of $\IN$, to construct a well-order $\preccurlyeq$ of order type $\w$ on the set $\mathbb N$ such that $\sum_{n\in\IN}^{\preccurlyeq}x_n=x+y$, which means that for any $\e>0$ there exists $k\in\IN$ such that for any $m\in\IN$ with $k\preceq m$ we get $\|\sum_{n\preccurlyeq m}x_n-(x+y)\|<\e$.

 Let $s_1\in X_1$ be the sum of the absolutely convergent series $\sum_{n\in\Omega}\pr_1(x_n)$. Since $y\in\Sigma_1$, there exists a permutation $\sigma:\IN\to\IN$ such that $y=\sum_{n=1}^\infty \pr_1(x_{\sigma(n)})$.
The permutation $\sigma$ determines a well-order $\preceq$ on the set $\Lambda:=\IN\setminus\Omega$ defined by $n\preceq m$ iff $\sigma(n)\le \sigma(m)$. For this well-order on $\Lambda$ we get $\sum_{n\in\Lambda}^\preceq x_n=y-s_1$. From now on, talking about minimal elements of subsets of $\Lambda$ we have in mind the well-order $\preceq$. The set $\Omega\subset\IN$ is endowed with the standard well-order $\le$.

By induction, we shall construct decreasing sequences of sets $(\Lambda_i)_{i\in\w}$ and $(\Omega_i)_{i\in\w}$ and an increasing sequence of finite sets $(F_i)_{n\in\w}$ such that the following conditions are satisfied for every $n\in\w$:
\begin{enumerate}
\item $\Lambda_{i+1}=\Lambda_i\setminus\{\min\Lambda_i\}$;
\item $\Omega_{i+1}\subset \Omega_i\setminus\{\min\Omega_i\}$;
\item $F_{i+1}=F_i\cup \{\min \Lambda_i\}\cup (\Omega_{i}\setminus\Omega_{i+1})$;
\item $\|x-\sum_{n\in F_{i+1}}\pr_0(x_n)\|<\frac1{2^i}$;
\item for any subset $E\subset F_{i+1}\setminus F_i$ we get $\|\sum_{n\in E}x_n\|\le C(\frac1{2^i}+\|x_{\min \Lambda_i}+x_{\min\Omega_i}\|)$.
    \end{enumerate}

We start the inductive construction, applying Claim~\ref{cl3} and choosing a finite set $F_0\subset\Omega$ such that $\|x-\sum_{n\in F_0}\pr_0(x_n)\|\le\|x-\sum_{n\in F_0}x_n\|<1$. Next, we put $\Omega_0=\Omega\setminus F_0$ and $\Lambda_0=\Lambda=\IN\setminus\Omega$.

Assume that for some $i\ge 0$ the sets $F_{i}$, $\Omega_{i}$ and $\Lambda_{i}$ have been constructed. Consider the element $a_i:=x-\pr_0(x_{\min\Lambda_i}+x_{\min\Omega_i}+\sum_{n\in F_{i}}x_n)\in X_0$ and observe that the inductive assumption (4) guarantees that
$$\|a_i\|\le \|x-\sum_{n\in F_i}\pr_0(x_n)\|+\|x_{\min\Lambda_i}+x_{\min\Omega_i}\|<\frac1{2^i}+
\|x_{\min\Lambda_i}+x_{\min\Omega_i}\|.$$
Applying Claim~\ref{cl3}, find a finite subset $E_{i}\subset \Omega_i\setminus (F_i\cup\{\min\Omega_i\})$ such that
$\|a_i-\sum_{n\in E_{i}}x_n\|<\frac1{2^{i+1}}$ and for every $E\subset E_i$
$$\big\|\sum_{n\in E}x_n\big\|< C\cdot\max\{\tfrac1{2^{i+1}},\|a_i\|\}\le C(\tfrac1{2^i}+\|x_{\min\Lambda_i}+x_{\min\Omega_i}\|).$$

Let $F_{i+1}=F_i\cup E_i\cup\{\min\Omega_i,\min\Lambda_i\}$, $\Omega_{i+1}=\Omega_i\setminus(E_i\cup\{\min\Omega_i\})$ and $\Lambda_{i+1}=\Lambda_i\setminus\{\min\Lambda_i\}$.
This completes the inductive step.

After completing the inductive construction of the sequence $(F_i)_{i\in\w}$, on the union $\bigcup_{i\in\w}F_i=\IN$ fix any well-order $\preccurlyeq$ such that $n\preccurlyeq m$ for any $i\in\w$ and numbers $n\in F_i$ and $m\notin F_i$. Observe that the well-order $\preccurlyeq$ induces the well-order $\preceq$ on the set $\Lambda$, which, combined with the absolute convergence of the series $\sum_{n\in\Omega}\pr_1(x_n)$, ensures that the series  $\sum_{n\in\IN}^{\preccurlyeq}\pr_1(x_n)$ converges to $y$. On the other hand, the conditions (4), (5) of the inductive construction and the convergence of  the sequence  $(x_n)_{n=1}^\infty$ to zero imply that the series $\sum_{n\in\IN}^\preccurlyeq \pr_0(x_n)$ converges to $x$. Consequently, the series $\sum_{n\in\IN}^\preccurlyeq x_n$ converges to $x+y$, which yields the desired inclusion $x+y\in \Sigma$.

Finally, observe that $\Sigma+\Gamma^\perp=(X_0+\Sigma_1)+(X_0+\Gamma_1^\perp)=
X_0+(\Sigma_1+\Gamma_1^\perp)=X_0+\Sigma_1=\Sigma$.

\section{Generalizing Lev\'y-Steinitz Theorem to Abelian topological groups}

The problem of extending the Lev\'y-Steinitz Theorem to infinite-dimensional Banach space was posed by Stefan Banach in the Scottish Book \cite[Problem 106]{SB} and to Abelian topological groups by Stanis\l aw Ulam , who noticed (without proof) in \cite{Ulam} that for any convergent series $\sum_{n=1}^\infty x_n$ in a compact metrizable Abelian topological group $X$ the set $\Sigma$ of potential sums of the series $\sum_{n=1}^\infty x_n$ coincides with the shift $H+\sum_{n=1}x_n$ of some subgroup $H\subset X$. This fact (whose proof was given in Banaszczyk \cite[10.2]{Ban}) can be considered as an extension of the first part of Lev\'y-Steinitz Theorem (describing the structure of the set $\Sigma$ of potential sums of a series). On the other hand, the second part of Lev\'y-Steinitz Theorem (characterizing series with non-empty set $\Sigma$ of partial sums) does not extend to locally compact Abelian topological groups as shown by the following simple example.

\begin{example}\label{ex} Let $\IT=\IR/\IZ$ be the circle group. The compact topological group $X=\IT\times\IT$ contains a sequence $(x_n)_{n\in\w}$ such that
\begin{enumerate}
\item $(x_n)_{n=1}^\infty$ converges to zero;
\item for every continuous homomorphism $f:X\to\IT$ the series $\sum_{n=1}^\infty f(x_n)$ is potentially convergent in $\IT$;
\item the series $(x_n)_{n=1}^\infty$ is not potentially convergent in $X$.
\end{enumerate}
\end{example}

\begin{proof} Let $q:\IR\to\IR/\IZ=\IT$ be the quotient homomorphism and $q^2:\IR^2\to\IT^2=X$ be the covering homomorphism defined by $q^2(x,y)=(q(x),q(y))$ for $(x,y)\in\IR^2$. The compactness of the topological group $X=\IT^2$ implies that the set of all non-zero continuous homomorphisms $h:X\to\IT$ is countable and hence can be enumerated as $\{h_n\}_{n\in\w}$. For every $n\in\w$ consider the continuous homomorphism $f_n=h_n\circ q^2:\IR^2\to\IT$. Since the space $\IR^2$ is simply-connected, the homomorphism $f_n$ can be lifted to a unique continuous homomorphism $\hat f_n:\IR^2\to\IR$ such that $q\circ\hat f_n=f_n$. Being continuous and additive, the homomorphism $\hat f_n$ is linear. Choose any non-zero linear continuous functional $f:\IR^2\to\IR$ such that $f\notin\{\hat f_n\}_{n\in\w}$. Then for every $n\in\w$ the sets $\{x\in \IR^2:f(x)>0,\;f_n(x)>0\}$ and $\{x\in \IR^2:f(x)>0,\;f_n(x)<0\}$ are not empty. This allows us to choose a sequence $(z_n)_{n=1}^\infty$ in the open half-plane $\{z\in \IR^2:f(z)>0\}$ such that
\begin{enumerate}
\item the sequence $(z_n)_{n=1}^\infty$ converges to zero;
\item the series $\sum_{n=1}^\infty f(z_n)$ diverges;
\item for every $k\in\w$ the series $\sum_{n=1}^\infty\max\{0,\hat f_k(z_n)\}$ and
    $\sum_{n=1}^\infty\min\{0,\hat f_k(z_m)\}$ are divergent. 
\end{enumerate}
In the group $\IT^2$ consider the sequence $(x_n)_{n\in\w}$ of points $x_n=q^2(z_n)$, $n\in\IN$. 
The conditions (1) and (2) imply that the series $\sum_{n=1}^\infty z_n$ is not potentially convergent in $\IR^2$ and hence the series $\sum_{n=1}^\infty x_n$ is not potentially convergent in $\IT^2$.
On the other hand, the condition (3) combined with the Riemann Rearrangement Theorem imply that for every $k\in\w$ the series $\sum_{n=1}^\infty \hat f_k(z_n)$ is potentially convergent in the real line and hence the series $\sum_{n=1}^\infty h_n(x_n)=
\sum_{n=1}^\infty q\circ \hat f_n\circ q^2(z_n)$ is potentially convergent in the circle $\IT$. Therefore, the sequence $(x_n)_{n\in\w}$ satisfies the conditions (1)--(3) of Example~\ref{ex}.
\end{proof}

%It would be interesting to know if the Lev\'y-Steinitz Theorem generalize to locally compact Abelian groups.

%\begin{problem} Is a series $\sum_{n=1}^\infty x_n$ is a locally compact topological Abelian group $X$ potentially convergent if for any continuous homomorphism $h:X\to\mathbb T$ to the circle group $\mathbb T=\{x\in\mathbb C:|z|=1\}$ the series $\sum_{n=1}^\infty f(x_{\sigma(n)})$ is potentially convergent in $\mathbb T$?
%\end{problem}

\section{Acknowledgements}

The author express his sincere thanks to Vaja Tarieladze for asking the problem and many valuable suggestions and to Volodymyr Kadets for his  comments on the original Steinitz proof (which is differs from the proof presented in this paper).

\end{document}